\documentclass[12pt]{article}
\usepackage[T1]{fontenc}
\usepackage[utf8]{inputenc}
\usepackage[english]{babel}
\usepackage{amsmath,amssymb,amsthm}
\usepackage[a4paper,margin=28mm]{geometry}
\usepackage[hidelinks]{hyperref}
\newtheorem{theorem}{Theorem}
\newtheorem{conjecture}{Conjecture}
\newtheorem{lemma}{Lemma}[section]
\newtheorem{corollary}{Corollary}[section]
\theoremstyle{definition}
\newtheorem{definition}{Definition}[section]
\theoremstyle{remark}
\newtheorem{remark}{Remark}[section]
\newcommand{\eps}{\varepsilon}
\renewcommand{\phi}{\varphi}
\renewcommand{\leq}{\leqslant}
\renewcommand{\geq}{\geqslant}
\renewcommand{\le}{\leqslant}
\renewcommand{\ge}{\geqslant}
\DeclareMathOperator{\id}{id}
\DeclareMathOperator{\sgn}{sgn}
\begin{document}
\author{Ilya I. Bogdanov, Fedor Petrov,\\ Anton Sadovnichiy, Fyodor Ushakov}
\title{On biregular multigraphs and matchings in tensor products of bipartite graphs}
\maketitle
\begin{abstract}
We answer a question of P.~Higgins on permutations and involutions in
binary arrays, arising in the study of regular finite semigroups.
For a binary rectangular matrix $M$, we prove that if there is a
permutation $\pi(x,y)=(f(x,y),g(x,y))$ of its cells satisfying
$M_{x,g(x,y)}=1$ for every cell, then there is an involution with the
same property. For an involution, the second cross entry
$M_{f(x,y),y}$ is automatically also equal to one, giving an affirmative
answer to Higgins's question under a weaker hypothesis. The elementary
proof places coins in the nonzero cells, reduces their support to a
forest, and assigns consecutive integer weights to the coins incident
to each row and column. We also prove a multiplicativity formula for bipartite expansion
and a perfect matching theorem for tensor products of supports
of matrices with prescribed margins. We study a stronger conjecture
of F.~V.~Petrov on disjoint swaps in biregular bipartite labeled
multigraphs. We prove this conjecture for $k=mn+\varepsilon$,
$|\varepsilon|\le2$, and for $n=6$, $k=6m+3$, and give a
group-theoretic reformulation.
\end{abstract}

\section{Introduction}

\begin{definition}
A \emph{bipartite labeled multigraph} is a tuple $(L,R,E,u,v)$,
where $L$, $R$, and $E$ are finite sets representing the left part,
the right part, and the edge set, respectively, and
$u:E\to L$ and $v:E\to R$ are the endpoint maps.
\end{definition}

The two parts are regarded as disjoint copies, even when their labels overlap.

\begin{definition}
A \emph{perfect matching} in a graph, possibly with loops, is an involution $\theta$ on its vertex set $V$ such that $v$ and $\theta(v)$ are adjacent for every $v\in V$. A fixed point is allowed only at a vertex with a loop, and a loop covers its vertex once. In a bipartite graph, the parts are disjoint and there are no loops. 
\end{definition}

For a nonnegative $n\times k$ matrix $M$, define the symmetric matrix
$B_M$, with rows and columns indexed by $[n]\times[k]$, by
\begin{equation}\label{BM}
 B_M\bigl((x,y),(x',y')\bigr)=M_{xy'}M_{x'y}.
\end{equation}
It is obtained from the usual tensor product $M\otimes M^{\mathsf T}$
by identifying its column index $(y',x')$ with $(x',y')$.
We use $B_M$ for the symmetric indexing and reserve $\otimes$ for the
usual tensor product.

A permutation $\pi$ of $[n]\times[k]$ is \emph{supported by $B_M$}
if $B_M(v,\pi(v))>0$ for every $v$.
Such a permutation is equivalent to a perfect matching in the bipartite
graph whose two parts are copies of $[n]\times[k]$ and whose adjacency
matrix is the support of $B_M$. A supported involution is a perfect
matching in the ordinary graph with the same symmetric adjacency
matrix, with loops allowed as above.

Higgins~\cite{H} asked whether the existence of a supported permutation
implies the existence of a supported involution when $M$ is binary.
Our first theorem answers this question affirmatively under the weaker
assumption that only one cross entry is prescribed for a permutation.

\begin{theorem}\label{HigginsT}
Let $M$ be a binary matrix with row set $X$ and column set $Y$, where
$|X|=n>0$ and $|Y|=k>0$. Suppose there is a permutation
\[
 \pi:X\times Y\longrightarrow X\times Y,\qquad
 \pi(x,y)=(f(x,y),g(x,y)),
\]
such that $M_{x,g(x,y)}=1$ for every $(x,y)\in X\times Y$.
Then there is an involution $\theta$ with the same property.
Moreover, $\theta$ is supported by $B_M$.
\end{theorem}

Indeed, if $\theta(x,y)=(x',y')$, the stated property gives
$M_{xy'}=1$. Applying it to $\theta(x',y')=(x,y)$ also gives
$M_{x'y}=1$. Thus the existence of a permutation satisfying the
one-cross-entry condition, the existence of a permutation supported
by $B_M$, and the existence of an involution supported by $B_M$
are equivalent.

The proof of Theorem~\ref{HigginsT} is given in
Section~\ref{sec:higgins-proof}.

\begin{corollary}\label{WeakBallsT}
Let $M$ be a nonnegative $n\times k$ matrix with all row sums equal to
$k$ and all column sums equal to $n$. Then $B_M$ supports an involution.
In particular, this holds for every such matrix with integer entries.
\end{corollary}

\begin{definition} \label {SquareBracket}
For maps $u:E\to L$ and $v:E\to R$, let $[u,v]$ denote the
weighted complete bipartite graph with parts $L,R$ in which the edge
$(x,y)$ has weight $|\{e\in E:u(e)=x,\ v(e)=y\}|$.
We use the same notation for the $L\times R$ matrix of these weights,
identifying the weighted graph with its matrix.
\end{definition}

Let $W_{L,R}$ denote the weighted complete bipartite graph on $L,R$
with every edge weight equal to $1$, or equivalently the corresponding
matrix of ones.

\begin{definition} \label {ActionA}
Let $A_{n,k}$ be the set of biregular bipartite labeled multigraphs
with $L=[n]$, $R=[k]$, and $E=[n]\times[k]$, in which each left
vertex has degree $k$ and each right vertex has degree $n$.
Let $S_{n,k}$ be the group of permutations of $[n]\times[k]$.
We identify an element of $A_{n,k}$ with its endpoint maps $(u,v)$
and define the following commuting left and right actions:
\[
 \sigma(u,v)=(u\circ\sigma^{-1},v),\qquad
 (u,v)\sigma=(u,v\circ\sigma).
\]
\end{definition}

\begin{definition} \label {ActionB}

For finite sets $L,R$, let $S_L,S_R$ be their permutation groups.
Their actions on matrices indexed by $L\times R$ are given by
\[
 (\alpha M)_{i,j}=M_{\alpha^{-1}(i),j},\qquad
 (\beta M)_{i,j}=M_{i,\beta^{-1}(j)}
 \quad(\alpha\in S_L,\ \beta\in S_R).
\]

\end{definition}

Let $B_{n,k}$ be the set of nonnegative integer $n\times k$ matrices
with row sums $k$ and column sums $n$. The assignment
$(u,v)\mapsto[u,v]$ defines a map $A_{n,k}\to B_{n,k}$.
For $G=(u,v)\in A_{n,k}$, we also write $[G]=[u,v]$.

Let $I=\{\iota\in S_{n,k}:\iota^2=\id\}$ be the set of involutions.

The second author proposed the following conjecture in~\cite{P}.

\begin{conjecture}\label{BallsC}
For every $(u,v)\in A_{n,k}$ there is an involution $\iota\in I$
such that $[u\circ\iota,v]=W_{[n],[k]}$.
\end{conjecture}

In~\cite{P}, the conjecture was stated as a problem about swapping
balls. Interpret the parts $L,R$ as $n$ children and $k$ colors, and each
edge $e$ as a ball of color $v(e)$ held by child $u(e)$. Each child has
$k$ balls, and there are $n$ balls of each color. A 2-cycle $(e,f)$ of
$\iota$ swaps the corresponding balls, while a fixed point represents
a ball that is not swapped. The conjecture asserts that the children
can exchange balls so that each child ends up with exactly one ball
of each color and every ball participates in at most one swap.

The conjecture is invariant under interchanging the two parts. Indeed,
for an involution $\iota$ the matrices $[v\circ\iota,u]$ and
$[v,u\circ\iota]$ agree, by the change of variable $e\mapsto\iota(e)$.
Thus a solution of the transposed problem is also a solution of the
original problem.

\begin{remark}
Conjecture~\ref{BallsC} is stronger than the integer case of
Corollary~\ref{WeakBallsT}. Indeed, take $(u,v)\in A_{n,k}$ with
$[u,v]=M$, and suppose that $[u\circ\iota,v]=W_{[n],[k]}$.
Then
\[
 \sigma:E\longrightarrow[n]\times[k],\qquad
 \sigma(e)=(u(\iota(e)),v(e))
\]
is a bijection. The involution $\theta=\sigma\circ\iota\circ\sigma^{-1}$
is supported by $B_M$: if $\sigma(e)=(x,y)$ and
$\sigma(\iota(e))=(x',y')$, the original edges $\iota(e)$ and $e$
show that $M_{xy'}>0$ and $M_{x'y}>0$.
The construction in Theorem~\ref{HigginsT} uses only the support of $M$;
it does not prescribe the multiplicity with which each original edge
type is used. Thus Conjecture~\ref{BallsC} remains a separate problem.
\end{remark}

\section{Main results}

We also establish the following results on tensor products and on
Conjecture~\ref{BallsC}.

\begin{theorem}\label{TensorMatchingT}
Let $n_1,k_1,n_2,k_2$ be positive integers with $n_1k_2=n_2k_1$.
Let $M_1$ be a nonnegative $n_1\times k_1$ matrix with row sums $k_2$
and column sums $n_2$, and let $M_2$ be a nonnegative $n_2\times k_2$
matrix with row sums $k_1$ and column sums $n_1$.
Then the bipartite graph whose adjacency matrix is the binary support
of $M_1\otimes M_2^{\mathsf T}$ has a perfect matching.
\end{theorem}

\begin{theorem}\label{BallsT}
Conjecture~\ref{BallsC} holds whenever $k=mn+\eps$, where
$\eps\in\{-2,-1,0,1,2\}$ and $m\in\mathbb N$.
\end{theorem}

\section{Proofs of the main results}

\subsection{Permutations and involutions}
\label{sec:higgins-proof}

\begin{lemma}\label{ForestCoins}
Let $F$ be a finite forest with a positive integer number of coins
on each edge. For a vertex $v$, let $Q(v)$ be the total number of
coins on the edges incident to $v$. One can assign an integer weight
$w(c)$ to each coin $c$ so that, at every vertex $v$, the incident
coins have distinct weights forming a set of $Q(v)$ consecutive
integers.
\end{lemma}

\begin{proof}
We use induction on the number of edges. The assertion is immediate
for a forest without edges. Choose an edge $e=uv$, and assign
consecutive integer weights to its coins. Remove $e$ and apply the
induction hypothesis to the remaining forest. The components
containing $u$ and $v$ are distinct. Shift all weights in the component
containing $u$ so that its block of weights at $u$ immediately precedes
the block on $e$; if $u$ is isolated, no shift is needed. Do the same
in the component containing $v$. These shifts preserve the condition
at every other vertex and establish it at $u$ and $v$ as well.
\end{proof}

\begin{proof}[Proof of Theorem~\ref{HigginsT}]
For each cell $(x,y)$, put a coin in the cell $(x,g(x,y))$.
Several coins may occupy the same cell. The resulting placement
satisfies
\begin{equation}\label{CoinMargins}
 \begin{gathered}
 \text{each row contains $k$ coins, each column contains $n$ coins,}\\
 \text{and every occupied cell has entry $1$ in $M$.}
 \end{gathered}
\end{equation}
The row counts follow from the placement rule. The column counts
follow because $\pi$ is a permutation: exactly $n$ image cells have
any prescribed second coordinate.

Among all placements satisfying~\eqref{CoinMargins}, choose one
with the fewest occupied cells. Form a bipartite graph on the rows
and columns, with an edge for each occupied cell. This graph is a
forest. Indeed, on a cycle, subtract the smallest coin count on
alternate edges from those edges and add it to the other edges.
This preserves all row and column totals, keeps all counts
nonnegative integers, and makes at least one occupied cell empty
without creating any new occupied cell, a contradiction.

Apply Lemma~\ref{ForestCoins} to this forest. Henceforth consider
all weights modulo $N=n+k$. Write the weights in a column $C$ and
a row $R$ as
\[
 I_C=\{a+1,\ldots,a+n\},\qquad
 I_R=\{b+1,\ldots,b+k\},
\]
respectively. Since $n,k<N$, every residue in either block belongs
to a unique coin incident to the corresponding vertex.
Exactly one of the following conditions holds:
\begin{equation}\label{CoinAlternatives}
 -a\in I_R,\qquad n-b\in I_C.
\end{equation}
Indeed, putting $t=-a-b$ modulo $N$, the first condition means
$t\in\{1,\ldots,k\}$, while the second means
$n+t\in\{1,\ldots,n\}$, or equivalently
$t\in\{k+1,\ldots,N\}$.

We define $\theta$ on the cell $R\cap C$.
Suppose first that $-a\in I_R$. The unique coin of weight $-a$
in row $R$ lies in some column $C_1$. Write
\[
 I_{C_1}=\{d+1,\ldots,d+n\}.
\]
Since $-a\in I_{C_1}$, we also have $-d\in I_C$:
both assertions say that $-a-d\in\{1,\ldots,n\}$.
Let the unique coin of weight $-d$ in column $C$ lie in row $R_1$.
Set
\[
 \theta(R\cap C)=R_1\cap C_1.
\]
The first condition of~\eqref{CoinAlternatives} also holds at
$R_1\cap C_1$, since $-d\in I_{R_1}$. Repeating the construction
selects first the same coin in $R_1\cap C$ and then the same coin
in $R\cap C_1$, and therefore returns to $R\cap C$.

Suppose instead that $n-b\in I_C$. The unique coin of weight
$n-b$ in column $C$ lies in some row $R_1$. Write
\[
 I_{R_1}=\{d+1,\ldots,d+k\}.
\]
Then $n-b\in I_{R_1}$ implies $n-d\in I_R$, since both conditions
mean $n-b-d\in\{1,\ldots,k\}$. Let the unique coin of weight
$n-d$ in row $R$ lie in column $C_1$, and again set
\[
 \theta(R\cap C)=R_1\cap C_1.
\]
At the image cell the second condition of~\eqref{CoinAlternatives}
holds, since $n-d\in I_{C_1}$. Repeating this construction selects
the same two coins in reverse order and returns to $R\cap C$.

Thus $\theta$ is an involution. In either case the two cross cells
$R\cap C_1$ and $R_1\cap C$ contain coins, so their entries in $M$
are both one by~\eqref{CoinMargins}. Hence $\theta$ has the required
property and is supported by $B_M$.
\end{proof}

\begin{proof}[Proof of Corollary~\ref{WeakBallsT}]
Every row of $B_M$ has sum $nk$, since
\[
 \sum_{x',y'}M_{xy'}M_{x'y}
 =\left(\sum_{y'}M_{xy'}\right)
  \left(\sum_{x'}M_{x'y}\right)=nk.
\]
By symmetry, all column sums are also $nk$. For a set $U$ of left
vertices in the bipartite graph of its support, the total weight
leaving $U$ is $nk|U|$ and is at most $nk|N(U)|$. Hall's theorem
\cite[Theorem~2.1.2]{D} therefore gives a supported permutation.
Apply Theorem~\ref{HigginsT}
to the binary support matrix of $M$.
\end{proof}

\subsection{Proof of Theorem~\ref{TensorMatchingT}}

In this section, all bipartite graphs are assumed to have nonempty left and right parts.

Here the tensor product of bipartite graphs is defined through their
bipartite adjacency matrices: if $G$ has parts $L,R$ and $H$ has
parts $L',R'$, then $G\otimes H$ has parts
$L\times L',R\times R'$, and $(x,x')$ is adjacent
to $(y,y')$ precisely when $xy\in E(G)$ and
$x' y'\in E(H)$.

\begin{definition}
\label{h(G)}
Let $G$ be a bipartite graph with parts $L,R$. For $X\subseteq L$,
let $N(X)$ be the set of vertices adjacent to at least one vertex
of $X$. Define
$$
  h(G)=\min_{\varnothing \neq X \subseteq L} \frac{|N(X)|}{|X|}.
$$
\end {definition}

\begin {lemma} \label {TensorWithKL}
Let $p,q$ be positive integers, and $G$ be a bipartite graph. Then the following conditions are equivalent:

(1) $h(G) \geq \frac{p}{q}$;

(2) In the graph $G \otimes K_{p,q}$, there exists a matching that covers its left part.
\end {lemma}

\begin {proof}
$(1)\Rightarrow(2)$. Assume (1). We verify Hall's condition for
$G\otimes K_{p,q}$. Let $L$ be the left part of $G$.

Consider $\varnothing \neq X \subseteq L \times [p]$. Let $X'$ be the image of $X$ in the projection of $L \times [p]$ onto $L$. Then $|X'| \geq \frac{|X|}{p}$. By (1) and Definition~\ref{h(G)}, we have
$$
  \frac{|N(X')|}{|X'|} \geq h(G) \geq \frac{p}{q}.
$$  

Thus $|N(X)| = |N(X') \times [q]| = q|N(X')| \geq p|X'| \geq |X|$, as desired.

$(2)\Rightarrow(1)$. Consider any $\varnothing \neq X \subseteq L$ and put $X' = X \times [p]$. Since there exists a matching in $G \otimes K_{p,q}$ covering its left part, we have $q|N(X)| = |N(X')| \geq |X'| = p|X|$. So, $\frac{|N(X)|}{|X|} \geq \frac{p}{q}$. Since this holds for any nonempty $X \subseteq L$, it follows that $h(G) \geq \frac{p}{q}$.
\end {proof}

\begin {lemma} \label {hIsMultiplicativeL}
For any bipartite graphs $G$ and $H$, $h(G \otimes H) = h(G)h(H)$.
\end {lemma}

\begin {proof} 
We start by proving that $h(G \otimes H) \leq h(G)h(H)$. Let $X$ and $Y$ be subsets of the left parts in~$G$ and~$H$ such that $h(G)=\frac{|N(X)|}{|X|}$ and $h(H)=\frac{|N(Y)|}{|Y|}$. Using $N(X \times Y) = N(X) \times N(Y)$, we see that
$$
  h(G \otimes H) \leq \frac{|N(X \times Y)|}{|X \times Y|} = \frac{|N(X)|}{|X|}\cdot \frac{|N(Y)|}{|Y|} = h(G)h(H).
$$

It remains to show that $h(G \otimes H) \geq h(G)h(H)$. If either factor has an isolated left vertex, both sides are zero. Otherwise write $h(G)=a/b$ and $h(H)=c/d$, where $a,b,c,d$ are positive integers. By Lemma \ref{TensorWithKL}, $G \otimes K_{a,b}$ and $H \otimes K_{c,d}$ have matchings covering their left parts. Then $(G \otimes K_{a,b}) \otimes (H \otimes K_{c,d}) \cong (G \otimes H) \otimes (K_{a,b} \otimes K_{c,d}) \cong (G \otimes H) \otimes K_{ac, bd}$ also has a matching that covers the left part. Using Lemma~\ref{TensorWithKL}, we see that $h(G \otimes H) \geq \frac{ac}{bd} = h(G)h(H)$.
\end {proof}

\begin{proof}[Proof of Theorem \ref{TensorMatchingT}]
Let $G,H$ be the bipartite graphs defined by the supports of
$M_1,M_2^{\mathsf T}$, respectively. The row and column sums imply
$h(G)=k_2/n_2$ and $h(H)=n_1/k_1=n_2/k_2$.
Lemma~\ref{hIsMultiplicativeL} gives $h(G\otimes H)=1$.
The two parts have equal size, so Hall's theorem gives a perfect matching.
\end {proof}

\subsection{Results related to Conjecture \ref{BallsC}}

\subsubsection{Balanced auxiliary labels and small remainders}

We first illustrate Theorem~\ref{BallsT} in the case $n=k$, using
the interpretation in terms of children and balls. We use the standard
fact that every regular bipartite multigraph of positive degree has
a perfect matching; see~\cite[Corollary~2.1.3]{D} for the graph version.
The same Hall argument applies to multigraphs: if the common degree
is $r>0$, counting edges with multiplicity gives
$r|S|\le r|N(S)|$ for every subset $S$ of the left part.

Consider an $n\times n$ board whose rows correspond to the children.
A perfect matching selects one ball from each child, with all selected
balls of distinct colors. Place these balls in the first column,
one in each child's row. After removing them, the remaining multigraph
is $(n-1)$-regular. Repeating the procedure fills the board so that
every column contains one ball of each color. Transposing the board
swaps balls in disjoint pairs and leaves exactly the $n$ diagonal
balls fixed. Child $i$ then receives the balls formerly in column $i$,
one of each color.

\begin{lemma} \label{Coloring}
For every bipartite multigraph $G$ and positive integer $m$, the edges
can be colored with $m$ colors so that, at each vertex, the numbers
of incident edges of any two colors differ by at most $1$.
\end{lemma}

\begin{proof}
First suppose $m=2$, with colors red and blue. Add one new vertex to
each part of $G$, join every original vertex of odd degree to the new
vertex in the opposite part, and, if necessary, join the two new
vertices. The degrees of the new vertices have the same parity, so
this makes every degree even. Decompose the resulting bipartite
multigraph into even cycles and color each cycle alternately red and
blue. Every vertex then has equal red and blue degrees. Removing the
added vertices and edges gives the required coloring, since each
original vertex loses at most one incident edge.

For general $m$, define the \emph{cost} of a coloring to be the sum,
over all vertices and colors, of the squares of the corresponding
colored degrees. Choose a coloring of minimum cost. If two colors
have degrees differing by more than one at some vertex, apply the
two-color construction to their subgraph. At each vertex the sum of
these two degrees is fixed, and their sum of squares is minimized
when they differ by at most one. The recoloring therefore decreases
the cost strictly, a contradiction.
\end{proof}

\begin{lemma} \label{windFromColoring}
Let $(L,R,E,u,v)$ be a biregular bipartite labeled multigraph with
$|L|=n$, $|R|=k$, and $|E|=nk$. Write $k=mn+\eps$, where $m,\eps$
are integers and $|\eps|<n$, and let $F$ be a set of size $n$.
Then there is a coloring $w:E\to F$ such that $[w,v]=W_{F,R}$ and
every entry of $[u,w]$ equals $m$ or $m+\sgn(\eps)$.
\end{lemma}

\begin {proof}
Apply Lemma~\ref{Coloring} with color set $F$. Each vertex of $R$
has degree $n$, so every color occurs there exactly once and
$[w,v]=W_{F,R}$. At each vertex of $L$, the colored degrees differ
by at most one and sum to $k$. They therefore have the stated values.
\end {proof}

\begin{remark}
In Lemma~\ref{windFromColoring}, if $\eps\ne0$, call the edges of
$[u,w]$ of weight $m+\sgn(\eps)$ \emph{exceptional}.
The graph of exceptional edges is $|\eps|$-regular; for $\eps=0$,
we take it to be empty.
\end{remark}

\begin{lemma}\label{Symmetrization}
Let $u,w:E\to[n]$ be maps such that the matrix $[u,w]$ is symmetric.
Then there is an involution $\iota$ of $E$ with $u\circ\iota=w$.
\end{lemma}

\begin{proof}
For $i\ne j$, the sets $\{e:u(e)=i,w(e)=j\}$ and
$\{e:u(e)=j,w(e)=i\}$ have the same size. Pair them by an arbitrary
bijection and let $\iota$ interchange the paired elements. Fix every
$e$ with $u(e)=w(e)$.
\end{proof}

\begin{proof}[Proof of Theorem~\ref{BallsT}]
The case $n=1$ is immediate.

For $n=2$, use Lemma~\ref{windFromColoring} with
$k=2\lfloor k/2\rfloor+\delta$, $\delta\in\{0,1\}$; the resulting
matrix is constant or has a perfect matching of exceptional entries,
and the argument below applies. Thus we may assume
$n>2$, and then $|\eps|<n$.
Choose $w:E\to[n]$ as in Lemma~\ref{windFromColoring}.
When $\eps=0$, the matrix $[u,w]$ is constant and hence symmetric.
For $|\eps|=1$, its exceptional edges form a perfect matching, which
can be made diagonal by relabeling the values of $w$.
For $|\eps|=2$, the exceptional graph is a disjoint union of even
cycles. Each cycle admits an involutory reflection interchanging its
two parts. Using these reflections to identify the two parts makes
its adjacency matrix symmetric. Thus in every case there is a
permutation $\alpha\in S_n$ for which $[u,\alpha\circ w]$ is symmetric.
By Lemma~\ref{Symmetrization}, there is an involution $\iota$ with
$u\circ\iota=\alpha\circ w$. Consequently,
\[
 [u\circ\iota,v]=[\alpha\circ w,v]
 =\alpha[w,v]=W_{[n],[k]}.
\]
\end{proof}

\subsubsection{The case of six rows}

For $n=6$, Theorem~\ref{BallsT} leaves only $k=6m+3$.
We prove this case by an elementary argument.

Consider a bipartite multigraph with six left vertices of degree
$2m+1$ and $6m+3$ right vertices of degree $2$, where degrees count
edges with multiplicity. Two distinct left vertices $u,v$ are
\emph{complementary} if $|N(\{u,v\})|=2m+1$.
Each left vertex has at most one complementary vertex.

A complementary pair forms a component on its two left vertices:
their $2(2m+1)$ incident edges exhaust the capacity of their $2m+1$
right neighbors. In particular, distinct complementary pairs are disjoint.

\begin{lemma}\label{find3}
In a multigraph as above, let $u,v,f$ be distinct left vertices such
that $u,v$ are not complementary. Then there is a left vertex
$w\notin\{u,v,f\}$ such that $|N(\{u,v,w\})|\ge3m+3$ and neither
$\{u,v,w\}$ nor its complement contains a complementary pair.
\end{lemma}

\begin{proof}
Let $l = |N(\{u,v\})|$. We consider several cases.

\emph{Case 1:} The four left vertices other than $u,v$ contain a
complementary pair. Choose $w$ from this pair with $w\ne f$.
Since complementary pairs are disjoint, $w$ is not complementary
to either $u$ or $v$. Moreover,
\[
|N(\{u,v,w\})| = |N(\{u,v\})| + |N(w)| \ge (2m+2) + (m+1) = 3m+3.
\]
There is at most one complementary pair among these four vertices:
two such components would leave $u,v$ as a third complementary pair,
contrary to the hypothesis. Every other complementary pair contains
$u$ or $v$, so the chosen triple and its complement split every pair.
In the remaining cases assume that Case~1 does not occur.

\emph{Case 2:} Both $u$ and $v$ have complementary vertices (say $u'$ and $v'$, respectively). Then we can choose $w$ to be any vertex distinct from $u,v,f$ that is not complementary to $u$ or $v$ (such a vertex exists because there are six vertices in total and at most two complements). Then
\[
|N(\{u,v,w\})| \ge |N(u)| + |N(v)| + |N(w)| \ge (m+1) + (m+1) + (m+1) = 3m+3.
\]

\emph{Case 3:} Only one of $u$ and $v$, say $u$, has a complementary vertex. Then we can take two different vertices $w_1$ and $w_2$ from the left part, different from $u,v,f$, that are not complementary to $u$. Note that $l = |N(\{u,v\})| \ge (m+1) + |N(v)|$. The total number of edges from $w_1$ and $w_2$ to $N(\{u,v\})$ is at most $2|N(v)| - (2m+1) < 2(2|N(v)| - 2m - 1) \le 2(2l - 4m - 3)$. Hence for at least one of $w_1$ and $w_2$, say $w_1$, the number of edges to $N(\{u,v\})$ is less than $2l - 4m - 3$, so the number of edges to the remaining vertices of the right part is greater than $6m+4 - 2l$. Consequently, $|N(\{u,v,w_1\})| \ge 3m+3$.

\emph{Case 4:} Neither $u$ nor $v$ has a complementary vertex.
Write $t=2l-4m-2\ge2$ for the number of edges from the other four
left vertices into $N(\{u,v\})$. If a vertex $s$ among those four
satisfies $|N(\{u,v,s\})|\le3m+2$, it has at most
$2(3m+2-l)$ edges outside $N(\{u,v\})$, and hence at least $t-1$
edges into it. Three such vertices would contribute
$3(t-1)>t$ edges, which is impossible. Thus at most two vertices are
unsuitable, and one can choose $w$ outside them and $\{u,v,f\}$.

In Cases~2 and~3, all complementary pairs contain $u$ or $v$, and
the construction avoids their partners when choosing $w$.
In Case~4 there are no complementary pairs. Thus in every case
neither the selected triple nor its complement contains such a pair.
\end{proof}

\begin{lemma}\label{SixRowsBase}
Conjecture~\ref{BallsC} holds for $n=6$, $k=6m+3$, where $m\ge0$.
\end{lemma}

\begin{proof}
First, color the edges with three auxiliary colors by
Lemma~\ref{Coloring}. In each color class every left vertex has degree
$2m+1$ and every right vertex has degree $2$. Call a triple $T$ of left
vertices \emph{admissible} if $|N(T)|\ge3m+3$ and neither $T$ nor
$T^c$ contains a complementary pair.

For every admissible $T$, this color class can be split into two
new colors, each occurring once at every right vertex, with degrees
$m+1$ on $T$ and $m$ on $T^c$ in the first new color. To prove this,
put $b_i=m+1$ on $T$ and $b_i=m$ on $T^c$. Hall's theorem applied
to $b_i$ copies of each left vertex requires precisely
\begin{equation}\label{SixHall}
 \sum_{i\in S}b_i\le |N(S)|\qquad(S\subseteq[6]).
\end{equation}
Since right degrees are two, $|N(S)|\ge
\lceil(2m+1)|S|/2\rceil$. This bound proves~\eqref{SixHall} for
$|S|=1,5,6$. For $|S|=2$, the only possible failure would be a
complementary pair contained in $T$. For $|S|=3$, only $S=T$
needs the stronger bound in the definition of admissibility.
For $|S|=4$, only $T\subseteq S$ needs attention: equality
$|N(S)|=4m+2$ would make every right neighbor of $S$ saturated by
edges from $S$. Its two-vertex complement $S^c\subset T^c$ would
then have exactly $2m+1$ neighbors and be a complementary pair,
which is excluded.
The empty set causes no restriction. The sum of the $b_i$ is $6m+3$,
so the matching selects one edge at every right vertex; assigning the
remaining edges the second new color proves the claimed split.

It remains to choose three admissible triples, one in each original
color class, in a form that permits symmetrization.

In the first color class, apply Lemma~\ref{find3} to any pair of
noncomplementary vertices and any distinct excluded vertex $f$.
This gives an admissible triple. Label its vertices $1,b,c$ so that
$b,c$ are not complementary in the second color class; this is
possible because each vertex has at most one complementary partner.
In the second color class, apply the lemma with $(f,u,v)=(1,b,c)$.
Label the resulting vertex $4$, giving the admissible triple $(b,c,4)$.

Now label $b,c$ as $2,3$ so that $3,4$ are not complementary in the
third color class. Apply the lemma there with $(f,u,v)=(2,3,4)$.
The resulting vertex is either $1$ or a new vertex, which we label $5$.
In the latter case, the admissible triples are
$(1,2,3)$, $(2,3,4)$, and $(3,4,5)$.
In the former case, apply the lemma again in the first color class
with $(f,u,v)=(1,2,3)$. This gives either vertex $4$ or a new vertex
$5$. The triples are then either
$(2,3,4)$, $(2,3,4)$, $(3,4,1)$ or
$(2,3,5)$, $(2,3,4)$, $(3,4,1)$.
Let $T_1,T_2,T_3$ be the resulting triples, and order the six
new colors as $T_1,T_2,T_3,T_1^c,T_2^c,T_3^c$, where a triple lists
the rows of degree $m+1$. In the three cases above, respectively,
the column orders
\[
 (1,6,5,4,3,2),\qquad (3,6,4,5,1,2),\qquad (1,4,3,2,5,6)
\]
make the $6\times6$ exceptional matrix symmetric. Lemma~\ref{Symmetrization}
then gives the required involution, exactly as in the proof of
Theorem~\ref{BallsT}.
\end{proof}

\section{Further discussion}

Conjecture~\ref{BallsC} remains open for general $n,k$.

\begin{remark}
We record without proof further cases in which
Conjecture~\ref{BallsC} is known to hold. It holds whenever
$\min(n,k)\le9$; the cases with smaller part of size seven, eight,
or nine have computer-assisted proofs using exact finite certificates.
It also holds when the bipartite graph of the positive entries of
$A=[u,v]$ is a forest.

\end{remark}

We give a group-theoretic reformulation of Conjecture~\ref{BallsC}.
Let $\Omega=[n]\times[k]$, and let $u_0,v_0$ be its coordinate
projections. Write $R=S_k^n$ for the subgroup preserving each row of
$\Omega$, and $C=S_n^k$ for the subgroup preserving each column.
Their normalizers in $S_{n,k}$ are
\[
 \widetilde R=S_k\wr S_n,\qquad
 \widetilde C=S_n\wr S_k.
\]
Recall that $I$ is the set of all involutions in $S_{n,k}$, including
the identity.

\begin{theorem}\label{GroupReformulation}
For fixed $n,k$, the following are equivalent:
\begin{enumerate}
\item Conjecture~\ref{BallsC} holds for every $(u,v)\in A_{n,k}$;
\item $ICR=S_{n,k}$;
\item $I\widetilde C\widetilde R=S_{n,k}$.
\end{enumerate}
\end{theorem}

\begin{proof}
By relabeling the edges, any $(u,v)\in A_{n,k}$ can be written with
$v=v_0$ and $u=u_0\circ\sigma^{-1}$ for some $\sigma\in S_{n,k}$.
This relabeling conjugates involutions, so it preserves the assertion
of Conjecture~\ref{BallsC}.

For a map $f:\Omega\to[n]$, the equality $[f,v_0]=W_{[n],[k]}$ holds
exactly when $f=u_0\circ c$ for some $c\in C$: in each column, $f$
takes every value in $[n]$ once. Hence the required involution $\iota$
exists exactly when
\[
 u_0\circ\sigma^{-1}\circ\iota=u_0\circ c
 \quad\text{for some }c\in C,
\]
or, equivalently, $\sigma^{-1}\iota=rc$ for some $r\in R$.
This is equivalent to $\sigma=\iota c^{-1}r^{-1}\in ICR$.
Thus (1) and (2) are equivalent.

Finally, $\widetilde C\widetilde R=CR$.
Indeed, write an element of $\widetilde C$ as $c\beta$, where $c\in C$
and $\beta$ permutes whole columns, and an element of $\widetilde R$
as $\alpha r$, where $r\in R$ and $\alpha$ permutes whole rows.
The permutations $\alpha$ and $\beta$ commute, while $\alpha\in C$
and $\beta\in R$. Therefore
\[
 (c\beta)(\alpha r)=(c\alpha)(\beta r)\in CR.
\]
The reverse inclusion is immediate, proving the equivalence with (3).
\end{proof}

\section{Acknowledgments}

We thank the Combinatorics Group at MIPT for organizing the 2019
and 2024 Adygeya problem-solving workshops, where much of the work
on this problem was carried out. The first author proved
the case $n=k$ at the 2019 workshop. Later at the same workshop,
he, Rom Pinchasi, and Ran Ziv proved the cases
$\eps\in\{-1,0,1\}$ of Theorem~\ref{BallsT}.
Narek Oganisyan of St.~Petersburg University subsequently gave a
different proof of the cases $|k-n|\le2$.

The work of F. Petrov and F. Ushakov was performed at the Saint Petersburg Leonhard Euler
International Mathematical Institute and was supported by the Ministry of Science and Higher
Education of the Russian Federation (agreement no. 075--15--2025--343).

\section*{Statement on the use of AI}

The proof of the involution theorem (Theorem~\ref{HigginsT}) was
obtained by OpenAI's ChatGPT and Codex. The proofs of the other
results proved in this paper were obtained by the authors.
ChatGPT and Codex also assisted with exposition, English-language
editing, and \LaTeX{} preparation. The authors are responsible for
checking the proofs and for the final text.

\end{document}